\documentclass[11pt]{amsart}
\usepackage{amsthm,amssymb,url,hyperref}
\usepackage{fullpage}
\usepackage{MnSymbol}
\usepackage[all]{xy}
\usepackage{color}
\usepackage{subfig} 

\usepackage{tikz, ulem}

\usepackage{tabularx}
\theoremstyle{plain}
\newtheorem{theorem}{Theorem}[section]

\newtheorem{corollary}[theorem]{Corollary}

\newtheorem{lemma}[theorem]{Lemma}

\newtheorem{proposition}[theorem]{Proposition}

\theoremstyle{definition}

\newtheorem{example}[theorem]{Example}
\newtheorem{question}[theorem]{Question}
\newtheorem{question*}{Question}

\newtheorem{remark*}{Remark}
\newtheorem{conjecture*}{Conjecture}

\newtheorem{remark}[theorem]{Remark}

\def\ker#1{\mathrm{ker}(#1)}

\def\lmlt{\mathrm{Inn}}
\def\dis{\mathrm{Trans}}

\def\nCS{n\mathcal{CS}}

\def\hom_P#1{Hom_\mathcal{P}(#1)}
\def\setof#1#2{\{#1\, : \,#2\}}

\def\Z{\mathbb Z}

\def\OC{\mathcal{OS}}

\def\dOC{t\mathcal{OS}}

\def\cg#1{\equiv_\alpha}
\newcommand*\xbar[1]{%
   \hbox{%
     \vbox{%
       \hrule height 0.5pt 
       \kern0.5ex
       \hbox{%
         \kern-0.1em
         \ensuremath{#1}%
         \kern-0.1em
       }%
     }%
   }%
} 

\title{Quandles with orbit series conditions}
\author{M. Bonatto}
\author{A. Crans}
\author{T. Nasybullov}
\author{G. Whitney}

\address[M. Bonatto]{IMAS--CONICET and Universidad de Buenos Aires}
\email{marco.bonatto.87@gmail.com}

\address[A. Crans]{Loyola Marymount University, Los Angeles, USA}
\email{acrans@lmu.edu}

\address[T. Nasybullov]{Sobolev Institute of Mathematics, Novosibirsk, Russia}
\email{ntr@math.nsc.ru}

\address[G. Whitney]{Studio Infinity, Los Angeles, USA}
\email{gwhitney@post.harvard.edu}

\begin{document}
\begin{abstract}
 
We introduce the notion of an {\it orbit series} in a quandle. Using this notion we define four families of quandles based on finiteness conditions on their orbit series.  
Intuitively, the classes $\dOC$ and $\dOC_n$  correspond to finitary compositions of trivial quandles while the classes $\OC$ and $\OC_n$ correspond to finitary compositions of connected quandles. We study properties of these four families of quandles and explore their relationships with several previously studied families of quandles:  reductive, $n$-reductive, locally reductive, $n$-locally reductive, and solvable quandles.

~

~
\end{abstract}
\maketitle

\section{Introduction}
 A quandle is an algebraic structure whose axioms are derived from the Reidemeister moves on oriented link diagrams. They were first introduced by Joyce \cite{Joyce} and Matveev \cite{Matveev} as an invariant for knots in $\mathbb{R}^3$. More precisely,  to each oriented diagram $D_K$ of an oriented knot $K$ in $\mathbb{R}^3$ one can associate the quandle $Q(K)$, which does not change if we apply the Reidemeister moves to the diagram $D_K$. The knot quandle $Q(K)$ can be constructed as the quotient of the free quandle $FQ_n$ on the $n$ arcs of $D_K$ by relations which can be obtained from the crossings of $D_K$. Joyce and Matveev proved that two knot quandles $Q(K_1)$ and $Q(K_2)$ are isomorphic if and only if there is an ambient isotopy that takes $K_1$ to $K_2$ or one that takes $K_1$ to the mirror image of the orientation reversal of $K_2$.  Over the years, quandles have been investigated by various authors in order to construct new invariants for knots and links (see, for example, \cite{Carter, Kamada, Nelson}).

The knot quandle is a very strong invariant for knots in $\mathbb{R}^3$; however, usually it is very difficult to determine if two knot quandles are isomorphic. Sometimes homomorphisms from knot quandles to simpler quandles provide useful information that helps determine whether two knot quandles are isomorphic. This potential utility leads to the necessity of studying broader classes of quandles 
from the algebraic point of view. Algebraic properties of quandles including their automorphisms and residual properties have been investigated, for example, in \cite{BDS, BNS, BSS, BiaBon, Clark, HosSha, JPSZ, NelWon, Nosaka}. A (co)homology theory for quandles and racks has been developed in \cite{Carter2, Fenn1, Fenn2, Nosaka2013} that led to stronger invariants for knots and links.  Many new constructions of quandles have been introduced, for example, in \cite{BarNasSingBiquandle, BarNas, BonCraWhi, Crans-Nelson}.

Quandles also find applications in the study of the set-theoretical Yang-Baxter equation. The Yang-Baxter equation first appeared in theoretical physics and statistical
mechanics in the works of Yang \cite{Yan} and Baxter \cite{Bax1,Bax2}. The notion of a set-theoretical solution of the Yang-Baxter equation was introduced by V.~Drinfel'd in the context of quantum groups (see \cite{Dri}) as a pair $(X, r)$,
where $X$ is a set and $r:X\times X\to X\times X$ is a bijective map such that 
$$(r\times id)(id \times r)(r\times id)=(id\times r)(r \times id)(id\times r).$$
If $Q$ is a quandle with operation $\rhd$, we obtain such a solution from the pair $(Q,r)$, where $r:Q\times Q\to Q\times Q$ is the map given by $r(x,y)=(x\rhd y,x)$
for $x,y\in Q$.

The recent paper \cite{BarNasmult1}  constructed a general method whereby a given solution of the set-theoretical Yang-Baxter equation on an arbitrary algebraic system $X$ can be used for constructing a representation of the virtual braid group $VB_n$ by automorphisms of the algebraic system  $X$. Also \cite{BarNasmult2} introduced a method to construct an invariant for virtual knots and links from a given solution of the set-theoretical Yang-Baxter equation on an algebraic system $X$. When $X$ is a quandle with the corresponding Yang-Baxter solution as above, then it is possible, using procedures described in \cite{BarNasmult1, BarNasmult2}, to construct a representation $VB_n\to {\rm Aut}(X)$ and a quandle invariant for virtual links.  However, in order to apply these procedures to $X$ and $r$, the quandle $X$ must satisfy certain specific conditions described in \cite{BarNasmult1, BarNasmult2}. In this regard, there is a natural problem of constructing specific families of quandles that are convenient to work with. 	

Turning to group theory for inspiration, convenient families of groups include those satisfying certain finite type conditions: (locally/residually) finite groups, (locally/residually) nilpotent groups, (locally/residually) solvable groups, etc. By analogy, it seems fruitful to consider similar families of quandles. This idea was partially implemented in \cite{BonSta} using commutator theory from universal algebra to introduce and study nilpotent and solvable quandles. Other previously studied families of quandles given by finite type conditions are reductive  quandles (see \cite{BonSta2}, and \cite{JPZ yoo} where they are called `multipermutational') and locally reductive quandles (see \cite{PilRom}, where this notion is considered for modes but not specifically named).  

Here we introduce the notion of an {\it orbit series} in a quandle. Using this notion, we define four families of quandles based on finiteness conditions on their orbit series. Intuitively, the classes $\dOC$ and $\dOC_n$  correspond to finitary compositions of trivial quandles, while the classes $\OC$ and $\OC_n$ correspond to finitary compositions of connected quandles. In particular, the class of quasi-trivial quandles discussed in \cite{QTri} corresponds exactly to $\dOC_2$ here. We study properties of these four families of quandles and explore their relationships with several previously studied families of quandles:  reductive, $n$-reductive, locally reductive, $n$-locally reductive, and solvable quandles.

The paper is organized as follows. Section~\ref{prem} summarizes the necessary preliminaries about quandles. In Section~\ref{redqua} we discuss the notions of reductive and locally reductive quandles and derive some of their properties. In Section~\ref{orbitserrrr} we introduce the notion of an orbit series in a quandle, give definitions of the families $\OC$, $\OC_n$, $\dOC$, $\dOC_n$, and provide several examples. Section~\ref{section about properties} is devoted to the study of properties of the families $\OC$, $\OC_n$, $\dOC$, $\dOC_n$ and of the quandles from these families: in Section~\ref{subquandles and so on} we study the behavior of the families $\OC$, $\OC_n$, $\dOC$, $\dOC_n$ under taking subquandles, homomorphic images, and direct products; in Section~\ref{are there connected subquandles?} we examine the subquandle structure of quandles from the families $\OC$, $\OC_n$, $\dOC$, $\dOC_n$; and in Section~\ref{orbmultred} we investigate connections between the families  $\dOC$, $\dOC_n$ and the families of reductive, locally reductive, and solvable quandles. 
\subsection*{Acknowledgements} T.~Nasybullov is supported by the Ministry of Science and Higher Education of the Russian Federation, government program of Sobolev Institute of Mathematics SB RAS, project  0250-2019-0001. Alissa S. ~Crans is supported by a grant from the Simons Foundation (\#360097, Alissa Crans).

\section{Preliminaries}\label{prem}

\subsection{Quandles} A quandle $Q$ is a set equipped with a binary operation $\rhd$ satisfying the following three axioms:
\begin{enumerate}
\item[(i)] $a\rhd a=a$ for all $a\in Q$,
\item[(ii)] for all $a \in Q$, the map $L_a:b\mapsto a\rhd b$ is a bijection of $Q$, and
\item[(iii)] $a\rhd ({b} \rhd c)=(a\rhd b)\rhd(a\rhd c)$ for all $a,b,c\in Q$.
\end{enumerate}
Axioms (ii) and (iii) are equivalent to the condition that each map $L_a:Q\to Q$  is an automorphism. The group ${\rm Inn}(Q)=\left\langle \setof{L_a}{a\in Q}\right\rangle$  generated by all $L_a$ is called the {\it inner automorphism group} of $Q$. The group $\dis{(Q)}=\left\langle \setof{L_a L_b^{-1}}{a,b\in Q}\right\rangle$ is a normal subgroup in ${\rm Inn}(Q)$ called the {\it transvection group} of $Q$ (sometimes called the displacement group of $Q$). Both groups ${\rm Inn}(Q)$ and $\dis{(Q)}$ act on $Q$. The orbit of an element $a\in Q$ under the action of ${\rm Inn}(Q)$ coincides with the orbit of $a$ under the action of $\dis{(Q)}$. This orbit is denoted by ${\rm Orb}(a,Q)$ and is called the {\it orbit} of $a$ in $Q$, or the component of $a$ in $Q$. If $Q={\rm Orb}(a,Q)$ for some $a\in Q$, then $Q$ is said to be {\it connected}.

The following are well-known examples of quandles. 

\begin{itemize}
\item The simplest example of a quandle is the {\it trivial quandle} on a set $X$:  $Q=(X,\rhd)$ where
$$x\rhd y=y$$
 for all $x,y\in X$. If $|X|=n$, then the trivial quandle on $X$ is denoted by $T_n$. Note that $T_1$, the terminal object in the category of quandles, will appear repeatedly.

\item Let $G$ be a group with $x, y \in G$. Denote the conjugate of $y$ by $x$ as $y^x=x^{-1}yx$. For an arbitrary integer $n$, the set $G$ with the operation 
$$x\rhd y=y^{x^n}=x^{-n}yx^n$$ 
defines a quandle called the {\it $n$-th conjugation quandle} of the group $G$, denoted by ${\rm Conj}_n(G)$. For the sake of simplicity, we will write ${\rm Conj}(G)$ for the first conjugation quandle ${\rm Conj}_1(G)$. Note that the quandle ${\rm Conj}_0(G)$ is trivial. If $H\subset G$ is closed under conjugation, then ${\rm Conj}(H)$ is the quandle on $H$ with the operation given by $g\rhd h=g^{-1}hg$ for $g,h\in H$. 

\item Let $\varphi$ be an automorphism of a group $G$. Then the set $G$ with the operation 
$$x\rhd y=\varphi(yx^{-1})x$$ forms a quandle called the {\it generalized Alexander quandle} of $G$ with
respect to $\varphi$, denoted by ${\rm Alex}(G,\varphi)$.  Alexander quandles were studied, for example, in \cite{BDS, ClaElhSaiYet, Clark} and in \cite{BONalone} under the name of principal quandles.

\item If $G$ is an abelian group then the inversion operation $-$ is an automorphism of $G$, and $\mathrm{Alex}(G,-)$ is known as the {\it Takasaki quandle} \cite{Tak} of $G$.  When $C_n$ is the cyclic group of order $n$, $D_n = \mathrm{Alex}(C_n,-)$ is called the {\it dihedral quandle} on $n$ elements.
\end{itemize}

\subsection{Congruences and homomorphisms}\label{conghom}

Let $\alpha$ be an equivalence relation on a set $Q$. We write $a\, \alpha\, b$ to mean that $a,b\in Q$ are $\alpha$-related.  We denote the class of $a$ with respect to $\alpha$ by $[a]_\alpha$ and the quotient set by $Q/\alpha=\setof{[a]_\alpha}{a\in Q}$. 

A {\it congruence} of a quandle $Q$ is an equivalence relation $\alpha\subseteq Q\times Q$ that respects the algebraic structure. Namely, the operation on $Q/\alpha$ defined as 
$$[a]_{\alpha} \rhd [b]_{\alpha}=[a\rhd b]_{\alpha}$$
is well defined and provides a quandle structure on $Q/\alpha$. Equivalently, congruences are the equivalence relations such that if $a\,\alpha\,b$ and $c\,\alpha\,d$, then $(a\rhd c)\,\alpha\,(b\rhd d)$, and if $a\rhd x=c$ and $b\rhd y=d$, then $x\,\alpha\,y$.

The canonical surjective map $a\mapsto [a]_{\alpha}$ is a quandle homomorphism. If $h:Q\rightarrow Q^\prime$ is a quandle homomorphism, then the equivalence relation $$\ker{h}=\setof{(a,b)\in Q \times Q}{h(a)=h(b)}$$
is a congruence of $Q$. By virtue of the second homomorphism theorem for general algebraic structures \cite{Ber}, there exists a one-to-one correspondence between congruences and homomorphic images. The congruences of a quandle form a lattice, denoted by ${\rm Con}(Q)$, whose minimum is the identity relation $0_Q=\{(a,a)~:~a\in Q\}$ and maximum is $1_Q=Q \times Q$. Then ${\rm Con}(Q/\alpha)$ is given by the congruences 
$$\beta/\alpha=\setof{([a]_{\alpha},[b]_{\alpha})\in Q/\alpha\times Q/\alpha}{a\,\beta\,b}$$ 
for every $\beta\in {\rm Con}(Q)$ such that $\alpha\subseteq\beta$. Where it will not cause confusion, we will simply write $[a]_{\beta / \alpha}$ to mean $\big[ [a ]_\alpha \big]_{\beta / \alpha}$.

A class $\mathcal{K}$ of quandles  is said to be {\it closed under homomorphic images} if $Q/\alpha$ belongs to $\mathcal{K}$ for every quandle $Q\in\mathcal{K}$ and every congruence $\alpha\in {\rm Con}(Q)$. If $\alpha$ is a congruence of $Q$ and $a\in Q$, then the class $[a]_{\alpha}$ is clearly a subquandle of $Q$.   A class $\mathcal{K}$ of quandles  is said to be {\it closed under extensions} if $Q \in \mathcal{K}$ whenever $Q/\alpha  \in \mathcal{K}$ and $[a]_{\alpha}\in \mathcal{K}$ for every $a\in Q$.

Any congruence $\alpha$ of a quandle $Q$ induces a surjective homomorphism of groups 
\begin{align*}\pi_{\alpha}:\lmlt{(Q)}&\rightarrow \lmlt{(Q/\alpha)} \; \; {\rm defined \, by}\\
L_a&\mapsto L_{[a]_{\alpha}},~a\in Q
\end{align*}
which restricts and corestricts to the transvection groups. In particular,  for every $a\in Q$ and every $h\in\lmlt{(Q)}$ the equation
\begin{equation} \label{piprop}
[h(a)]_\alpha=\pi_\alpha(h)([a]_\alpha)
\end{equation} 
holds. Hence,
\begin{equation}\label{image of an orbit}
\setof{[h(a)]_\alpha}{h\in N}=\setof{\pi_\alpha(h)([a]_\alpha)}{h\in N}    
\end{equation}
for every $a\in Q$ and every $N\leq\lmlt{(Q)}$. 

We define the {\it transvection group relative to $\alpha$} as
$$\dis_{\alpha}=\langle \setof{L_a L_b^{-1}}{ a\,\alpha\,b}\rangle.$$
The set of all normal subgroups of the group $\lmlt{(Q)}$ contained in  $\dis(Q)$ is denoted by ${\rm Norm}(Q)=\setof{N\leq \dis(Q)}{N\trianglelefteq \lmlt{(Q)}}$. In particular, $\dis_\alpha \in {\rm Norm}(Q)$.  

If $N$ belongs to ${\rm Norm}(Q)$,  we define the following congruence of $Q$ as in \cite{BonSta} by:
\begin{align*}
	\mathcal{O}_N&=\setof{(a,b)\in Q \times Q}{b=n(a),~\text{for some}~n\in N},
\end{align*}
i.e., $\mathcal{O}_N$ is the orbit decomposition with respect to the action of $N$. The $\mathcal{O}_N$ decomposition will play a significant role in Section \ref{nredqua}.

From the third quandle axiom, it follows that $L_{x\rhd y}=L_x L_y L_x^{-1}$ for $L_x,L_y\in {\rm Inn}(Q)$. Thus, the map $L:x\mapsto L_x$ is a quandle homomorphism from $Q$ to ${\rm Conj}_{-1}({\rm Inn}(Q))$, with image its subquandle $L(Q) = \setof{L_x}{x\in Q}$. The kernel of the homomorphism $L$ is a congruence denoted by $\lambda_Q$, i.e., $x\,\lambda_Q\,y$ if and only if $L_x=L_y$. If $L$ is injective then $Q$ is said to be {\it faithful}. Note that $\dis_\alpha=1$ if and only if $\alpha\subseteq \lambda_Q$.  For any quandle we have the following inductively defined chain of homomorphic images: 
$$L_0(Q)=Q, \quad L_{n+1}(Q)=L(L_n(Q)).$$ 

We define the {\it enveloping group} of a quandle $Q$ in terms of generators and relations as
$$G_Q=\langle e_x \, | \, e_x e_y e_x^{-1}=e_{x\rhd y}, \, x,y\in Q\rangle.$$
The canonical map $i: x\mapsto e_x$ is a quandle homomorphism from $Q$ to Conj$(G_Q)$ with the image $E(Q)=\setof{e_x}{x\in Q}$. The map $G_Q\rightarrow \lmlt{(Q)}$ given by $e_x\mapsto L_x$ on the generators can be extended to a surjective group homomorphism with central kernel \cite[Proposition 2.37]{Eiss}.

\subsection{Abelian, nilpotent and solvable quandles} The properties of abelianess, nilpotence and solvability for quandles, in the sense of commutator theory \cite{fes}, have been investigated in \cite{BonSta}.

For an algebraic structure $A$, the commutator is a binary operation on the congruence lattice, defined using the centralization as follows. Let $\alpha$, $\beta$, and $\delta$ be congruences of $A$.   We say that {\it $\alpha$ centralizes $\beta$ over $\delta$} if for every $(n+1)$-ary term operation $t$ (i.e., any function $t:A^{n+1}\longrightarrow A$ obtained by an expression in the basic operations), every pair $a\,\alpha\,b$, and every $u_1\,\beta\,v_1$, $\dots$, $u_n\,\beta\,v_n$ we have
\begin{equation*}\label{term condition}
 t(a,u_1,\dots,u_n)\,\delta \, t(a,v_1,\dots,v_n)\quad\text{implies}\quad t(b,u_1,\dots,u_n) \, \delta \, t(b,v_1,\dots,v_n).
\end{equation*}
The {\it commutator} of $\alpha$ and $\beta$, denoted by $[\alpha,\beta]$, is the smallest congruence $\delta$ such that this implication holds. A congruence $\alpha$ is called {\it abelian} if $[\alpha,\alpha]=0_A$ and {\it central} if $[\alpha,1_A]=0_A$.

An algebraic structure $A$ is called {\it abelian} if the congruence $1_A$ is abelian. It is called {\it nilpotent} (resp. {\it solvable}) if and only if there is a chain of congruences 
\begin{displaymath}
    0_A=\alpha_0\leq\alpha_1\leq\ldots\leq\alpha_n=1_A
\end{displaymath}
such that $\alpha_{i+1}/\alpha_{i}$ is a central (resp. abelian) congruence of $A/\alpha_{i}$ for all $i\in\{0,1,\dots,n-1\}$. The length of the smallest such series is called the {\it length} of nilpotence (resp. solvability).

As for groups, one can define the series
\begin{displaymath}
    \gamma_{0}(A)=1_A,\qquad \gamma_{i+1}(A)=[\gamma_{i}(A),1_A],
\end{displaymath}
and
\begin{displaymath}
    \gamma^{0}(A)=1_A,\qquad \gamma^{i+1}(A)=[\gamma^{i}(A),\gamma^{i}(A)],
\end{displaymath}
and prove that an algebraic structure $A$ is nilpotent (resp. solvable) of length $n$ if and only if $\gamma_{n}(A)=0_A$ (resp. $\gamma^n(A)=0_A$). 

In the case of quandles, these properties turn out to be completely determined by group theoretic properties of the transvection group. Recall that a group acting on a set is {\it semiregular} if the pointwise stabilizers are trivial.

\begin{theorem}\cite{BonSta} \label{solvability_for_quandles}
	A quandle $Q$ is solvable (resp. nilpotent) if and only if $\dis(Q)$ is solvable (resp. nilpotent).  In particular, $Q$ is abelian if and only if $\dis(Q)$ is abelian and semiregular.
\end{theorem}

A quandle $Q$ is is said to be {\it medial} if   
$$(a\rhd b)\rhd (c\rhd d)=(a\rhd c)\rhd (b\rhd d)$$
for every $a,b,c,d\in Q$. Note that Joyce calls such quandles ``abelian" \cite{Joyce}. We choose to use the word ``medial'' (as in \cite{JPSZ}) to avoid conflict with the term ``abelian" as used in Theorem~\ref{solvability_for_quandles}  above. A quandle is medial if and only if its transvection group is abelian. Therefore all abelian quandles are medial; however, due to Theorem \ref{solvability_for_quandles}, the converse is false in general.

\section{Reductive and locally reductive quandles}\label{redqua}

 \subsection{n-Reductive quandles}\label{nredqua} Let $n$ be a positive integer. A quandle $Q$ is said to be {\it n-reductive} if either of the following two equivalent identities holds:
\begin{align}\label{red and med}
((\ldots((a\rhd c_{1})\rhd c_{2})\ldots )\rhd c_{n-1})\rhd c_n&=
((\ldots((b\rhd c_{1})\rhd c_{2})\ldots )\rhd c_{n-1})\rhd c_n\\
\notag ((\ldots((a\rhd c_{1})\rhd c_{2})\ldots )\rhd c_{n-1})\rhd c_n&=
((\ldots(c_1\rhd c_{2})\ldots )\rhd c_{n-1})\rhd c_n
\end{align}
for all $a,b,c_1,c_2,\dots,c_n \in Q$.  Note that every 1-reductive quandle is trivial.

Let the set of all $n$-reductive quandles be denoted by $\mathcal{R}_n~(n\geq 1)$ and let $\mathcal{R}_{\omega}=\bigcup_{n\geq 1}\mathcal{R}_n$.
The property of being reductive (i.e., belonging to $\mathcal{R}_{\omega})$ can be characterized in terms of the following inductively defined descending chain of congruences:

\[\mathcal{O}_Q^0=1_Q, \qquad \qquad \mathcal{O}_Q^{n+1}=\mathcal{O}_{\dis_{\mathcal{O}^n_Q}}.\]
Proposition \ref{strong red and inn} will provide equivalent definitions of reductive quandles, but first we need a technical lemma.

\begin{lemma}\label{factor of orbits}
Let $Q$ be a quandle and $\alpha\in \textrm{Con}(Q)$. Then $$[a]_{\mathcal{O}_{Q/\alpha}^n}=\setof{[b]_\alpha}{b\;\mathcal{O}_{Q}^n\; a}.$$
\end{lemma}

\begin{proof}
We proceed by induction on $n$; the case $n = 0$ is trivial.  Suppose the inductive hypothesis holds for $n$, i.e.,  
\begin{equation}
\label{useimmediately}
\setof{[b]_\alpha}{[b]_\alpha \;\mathcal{O}_{Q/\alpha}^n\; [a]_\alpha} = \setof{[b]_\alpha}{b\;\mathcal{O}_{Q}^n\; a}.
\end{equation}
We first observe that equation (\ref{useimmediately}) directly above means $\pi_\alpha (\dis_{\mathcal{O}_{Q}^n})=\dis_{\mathcal{O}_{Q/\alpha}^n}$: The group on the right-hand side is generated by elements of the form $L_{[a]_\alpha} L_{[b]_\alpha}^{-1}$ where $[b]_\alpha \;\mathcal{O}_{Q/\alpha}^n\; [a]_\alpha .$  But $L_{[a]_\alpha} L_{[b]_\alpha}^{-1} = \pi_\alpha (L_a L_b ^{-1})$, which is a generator of $\dis_{\mathcal{O}_{Q}^n}$ since $b \; \mathcal{O}^n_Q \; a$. 

For convenience, let $A = [a]_{\mathcal{O}^{n+1}_{Q/\alpha}}.$  By definition, $A =  \setof{m([a]_\alpha)}{m \in \dis_{\mathcal{O}^n_{Q/\alpha}}}$.  In light of the preceding observation, $A = \setof{\pi_\alpha(h)([a]_\alpha)}{h \in \dis_{\mathcal{O}^n_{Q}}}$.  Now, applying equation \eqref{image of an orbit} gives $A = \setof{[h(a)]_\alpha}{h \in \dis_{\mathcal{O}^n_{Q}}}$.  Finally, chasing the definitions gives $$[a]_{\mathcal{O}^{n+1}_{Q/\alpha}} = A = \setof{b_\alpha}{b \; \mathcal{O}_{\dis_{\mathcal{O}^n_Q}} \; a} =  \setof{[b]_\alpha}{b \; \mathcal{O}_{Q}^{n+1} \; a},$$ as desired. 
\end{proof}


\begin{proposition}\label{strong red and inn}
	Let $Q$ be a quandle. Then the following are equivalent.
	\begin{enumerate}
		\item[(i)] $Q\in\mathcal{R}_n$.
		\item[(ii)] $L_n(Q) =T_1$, the trivial quandle with one element.
		\item[(iii)] $\lmlt{(Q)}$ is nilpotent of nilpotency class $n-1$.
		\item[(iv)] $\mathcal{O}_Q^n=0_Q$.
	\end{enumerate}

\end{proposition}

\begin{proof}
	The equivalence between (i), (ii) and (iii) is Theorem 4.1 of  \cite{JPZ yoo}. We will show (i) is equivalent to (iv) using induction on $n$. In this proof we abbreviate $\lambda_Q$ by $\lambda$.
	
	(i) $\Rightarrow$ (iv): The basis of induction ($n=1$) is simple: in this situation $\dis(Q)=1$ and therefore $\mathcal{O}_Q^1=0_Q$. For the inductive step, let $Q$ be an $(n+1)$-reductive quandle. Then $Q/\lambda= L(Q)$ is $n$-reductive  and, by induction, $\mathcal{O}_{Q/\lambda}^{n} =0_{Q/\lambda}$. By Lemma  \ref{factor of orbits}, $\setof{[b]_\lambda}{b \; \mathcal{O}_Q^{n} \;a} = \big[ [a]_\lambda \big]_{\mathcal{O}_{Q/ \lambda}^{n}} = \{[a]_\lambda\}$. Thus, if $b\; \mathcal{O}_Q^{n} \;a$, the class $[b]_\lambda = [a]_\lambda,$ i.e., $b \; \lambda \; a$. In other words,  $\mathcal{O}_Q^{n}\subseteq \lambda$ and so $\dis_{\mathcal{O}^{n}_Q}\leq \dis_{\lambda}=1$. Hence $\mathcal{O}_Q^{n+1}=0_Q,$ completing the induction.
	
	(iv) $\Rightarrow$ (i): If $n=1$,  then  $\mathcal{O}_Q^1=0_Q$ means that  $Q$ is trivial and therefore $Q$ is $1$-reductive.  Let $\mathcal{O}^{n+1}_Q=0_Q$ and $Q^\prime=Q/\mathcal{O}^{n}_Q$.  The group homomorphism $\pi_{\mathcal{O}_Q^n}$ maps $\dis_{\mathcal{O}^k_Q}$ onto $\dis_{\mathcal{O}^{k}_Q/\mathcal{O}^{n}_Q}$ and then, by Lemma \ref{factor of orbits}, $\mathcal{O}^k_{Q^\prime}=\mathcal{O}_Q^k/\mathcal{O}_Q^n$ for every $k\leq n$.  Therefore $Q^\prime$ has a chain of congruences
	$$1_{Q^\prime}=\mathcal{O}_{Q^\prime}^0\subseteq \mathcal{O}^1_{Q^\prime}\subseteq \ldots\subseteq \mathcal{O}^{n-1}_{Q^\prime}\subseteq \mathcal{O}^{n}_{Q^\prime}=\mathcal{O}^{n}_Q/\mathcal{O}^{n}_Q= 0_{Q^\prime}.$$
	By induction, $Q^\prime$ is $n$-reductive.  Since $\mathcal{O}^{n+1}_Q=0_Q$, the orbits of $\dis_{\mathcal{O}^{n}_Q}$ are trivial, i.e., $\dis_{\mathcal{O}^{n}_Q}=1$. According to \cite{JPZ yoo}, $\mathcal{O}^{n}_Q\subseteq \lambda$. Thus, since $Q/\mathcal{O}^{n}_Q$ is $n$-reductive and $\mathcal{O}^{n}_Q\subseteq \lambda$, $Q$ is $(n+1)$-reductive by \cite{BonSta2}.
\end{proof}

\begin{corollary} \label{trivialreductive}
Let $Q$ be a reductive quandle.  If $Q$ is either faithful or connected, then $Q = T_1$.
\end{corollary}

\begin{proof}
In the faithful case, $L$ is injective, so $|Q| = 1$ by Proposition \ref{strong red and inn} (ii).  When $Q$ is connected, it is not difficult to see that all of the $\mathcal{O}^n _Q = 1_Q$.   Thus $1_Q = 0_Q$ by Proposition \ref{strong red and inn} (iv), so again, $Q = T_1$.
\end{proof}

\begin{corollary}\label{red iff nilpotent}
	Let $G$ be a group. Then ${\rm Conj}(G)\in \mathcal{R}_n$ if and only if $G$ is nilpotent of class $n$.
\end{corollary}
\begin{proof}
The statement follows from Proposition~\ref{strong red and inn} due to the equality $\lmlt({\rm Conj}(G))=G/Z(G)$.
\end{proof}

\subsection{Locally reductive quandles} Let $n$ be a positive integer. A quandle $Q$ is said to be {\it $n$-locally reductive} if  the equality 
\begin{equation}\label{red}
(\ldots ((a\rhd \underbrace{b )\rhd b)\rhd\ldots)\rhd b}_n=b
\end{equation}
holds for all $a,b\in Q$. A $1$-locally reductive quandle is trivial. We denote the class of all $n$-locally reductive quandles by $\mathcal{LR}_n~(n\geq 1)$ and let $\mathcal{LR}_{\omega}=\bigcup_{n\geq 1}\mathcal{LR}_n$. Every $n$-reductive quandle is obviously $n$-locally reductive; therefore, we have the inclusions:
\begin{equation}\label{MRinclusions}
\begin{matrix}\mathcal{LR}_{n-1} &\subset& \mathcal{LR}_n&\subset&\mathcal{LR}_{\omega}\\
\cup& &\cup&&\cup\\
\mathcal{R}_{n-1} &\subset& \mathcal{R}_n&\subset&\mathcal{R}_{\omega}
\end{matrix}
\end{equation}

If $Q$ is a medial quandle, then \eqref{red and med} and \eqref{red} are equivalent by \cite{JPSZ}. That is, $Q\in\mathcal{R}_n$ if and only if $Q\in\mathcal{LR}_n$. We want to understand connections between the classes $\mathcal{R}_{\omega}$ and $\mathcal{LR}_{\omega}$ in general.
\begin{lemma}\label{R is cloded for}
	The class $\mathcal{LR}_{\omega}$ is closed under taking subquandles, homomorphic images, finite direct products, and extensions. Moreover, if $Q/\alpha\in \mathcal{LR}_n$ and $[a]_\alpha\in \mathcal{LR}_m$ for all $a\in Q$, then $Q\in\mathcal{LR}_{n+m}$.
\end{lemma}	

\begin{proof}
	Each class $\mathcal{LR}_n$ is a variety, so clearly $\mathcal{LR}_{\omega}$ is closed under taking subquandles, homomorphic images, and finite direct products since the classes $\setof{\mathcal{LR}_n}{n\in \mathbb{N}}$ form a chain. Let $Q/\alpha\in \mathcal{LR}_n$ and $[a]_\alpha\in \mathcal{LR}_m$ for every $a\in Q$. Then 
	$$[(\ldots ((a\rhd \underbrace{b )\rhd b)\rhd\ldots)\rhd b}_n]_\alpha=[b]_\alpha$$
	and so 
	$$(\ldots ((a\rhd \underbrace{b )\rhd b)\rhd\ldots)\rhd b}_n)\rhd\underbrace{ b)\rhd\ldots)\rhd b}_{m}=b$$
	for every $a,b\in Q$.
\end{proof}

\begin{corollary}\label{reductivity by component}
	Let $Q$ be a quandle. If all the orbits of $Q$ are $n$-locally reductive, then $Q$ is $(n+1)$-locally reductive.
\end{corollary}
\begin{proof}
	The quotient $Q/\mathcal{O}_{\dis (Q)}$ is trivial, i.e., $1$-locally reductive. Hence we can apply Lemma \ref{R is cloded for}.
\end{proof} 

Let $G$ be a group with $a,b\in G$. Following standard notation, we use $[b,_{k}a]$ to denote the following inductively defined element of~$G$:
$$ [b,_{0}a] =a, \qquad   [b,_{n+1}a]=[b,[b,_na]] $$
for $n\in{\mathbb{N}}$. A subset $H$ of $G$ is called an {\it $n$-Engel subset} if $[b,_n a]=1$ for every $a,b\in H$. An element $g\in G$ is an {\it $n$-Engel element} of $G$ if $[b,_n g]=1$ for every $b\in G$.  The group $G$ itself is called {\it $n$-Engel} if it is an $n$-Engel subset of itself, or, equivalently, if every element is $n$-Engel.
\begin{proposition}\label{reductivity and Engel prop}
	Let $G$ be a group and $H$ be a subset of $G$ closed under conjugation. Then the following statements are equivalent.
	\begin{enumerate}
		\item[(i)] ${\rm Conj}(H)$ is $n$-locally reductive. 
		\item[(ii)] $H$ is an $n$-Engel subset of $G$. 
	\end{enumerate}
	In particular, ${\rm Conj}(G)$ is $n$-locally reductive if and only if $G$ is an $n$-Engel group. 
\end{proposition}
\begin{proof}
	We will show, by induction on $n$, that
	\begin{displaymath}
	(\ldots ((a\rhd \underbrace{b )\rhd b)\rhd\ldots)\rhd b}_n=[b,_{n-1} a]^{-1}b[b,_{n-1}a]
	\end{displaymath}
	for all $a, b \in G$. The basis of this induction is simply the definition of the quandle operation in ${\rm Conj}(H)$.  For the induction step, 
	\begin{align}
	\notag (\ldots ((a\rhd \underbrace{b )\rhd b)\rhd\ldots)\rhd b}_{n+1}&=(\ldots (((a\rhd \underbrace{b )\rhd b)\rhd\ldots)\rhd b)}_n\rhd b\\
	\notag&=[b,_{n-1} a]^{-1}b^{-1}[b,_{n-1}a]bbb^{-1}[b,_{n-1} a]^{-1}b[b,_{n-1}a]\\
	\notag&=[b,_{n} a]^{-1}b[b,_{n}a].
	\end{align}
Hence, $H$ is $n$-locally reductive if and only if $b=[b,_{n-1} a]^{-1}b[b,_{n-1}a]$ for all $a, b \in H$.  The latter equation is equivalent to $[b,_{n} a]=1$, so $H$ is an $n$-Engel subset.
\end{proof}

Example 18.3.1 of \cite{KarMer} gives a $3$-Engel group that is not nilpotent of any order.  This example, together with Proposition~\ref{reductivity and Engel prop}, Corollary \ref{red iff nilpotent}, and Example \ref{example1} of a $2$-locally reductive, non $2$-reductive quandle, show that the inclusions in \eqref{MRinclusions} are all strict. 

It turns out that Proposition \ref{reductivity and Engel prop} is sufficient to determine if any quandle is locally $n$-reductive as we next show.

\begin{proposition}\label{innred}
	Let $Q$ be a quandle. Then the following statements are equivalent.
	\begin{enumerate}
		\item[(i)] $Q\in \mathcal{LR}_{\omega}$.
		\item[(ii)] $E(Q)$ is an $n$-Engel subset of $G_Q$ for some $n\in \mathbb{N}$.
		\item[(iii)] $L(Q)$ is an $n$-Engel subset of $\lmlt{(Q)}$ for some $n\in \mathbb{N}$.
	\end{enumerate}
	In particular, if $L(Q)$ is $n$-locally reductive, then $Q$ is $(n+1)$-locally reductive.
\end{proposition}
\begin{proof} 
	(i) $\Rightarrow$ (ii) $\Rightarrow$ (iii) Let $Q$ be a quandle and $E(Q)=\setof{e_x}{x\in Q}$. Then the mappings
\begin{align*}
\begin{matrix}
Q\\
x
\end{matrix}~
\begin{matrix}
\to\\
\mapsto
\end{matrix}~
\begin{matrix}
E(Q)\\
e_x
\end{matrix}~
\begin{matrix}
\to\\
\mapsto
\end{matrix}~
\begin{matrix}
L(Q)\\
L_x
\end{matrix}
\end{align*}
	are quandle homomorphisms. Hence if 
	$Q$ (resp. $E(Q)$) is a $n$-locally reductive quandle, then $E(Q)$ (resp. $L(Q)$) is a $n$-locally reductive conjugation quandle. Therefore, by Proposition~\ref{reductivity and Engel prop}, $E(Q)$ (resp. $L(Q)$) is an $n$-Engel subset of $G_Q$ (resp. $\lmlt{(Q)}$). 
	
	(iii) $\Rightarrow$ (i) The quandle $L(Q)=Q/\lambda_Q$ is a $n$-locally reductive conjugation quandle and the equivalence classes of $\lambda_Q$ are trivial as quandles. Thus, by Lemma \ref{R is cloded for}, the quandle $Q$ is $(n+1)$-locally reductive.
\end{proof}

\begin{corollary}\label{n-Engel -> reductive}
	Let $Q$ be a quandle. If $\lmlt{(Q)}$ is $n$-Engel, then $Q$ is $(n+1)$-locally reductive.
\end{corollary}

\begin{proof}
	The quandle $L(Q)$ embeds into the $n$-locally reductive conjugation quandle ${\rm Conj}(\lmlt{(Q)}),$ and therefore it is $n$-locally reductive. Then we can apply Proposition \ref{innred}.
\end{proof}

Let $G$ be a finite group. Then $G$ is $n$-Engel for some $n$ if and only if $G$ is nilpotent \cite{Zorn}. So, ${\rm Conj}(G)\in\mathcal{R}_{\omega}$ if and only if  ${\rm Conj}(G)\in\mathcal{LR}_{\omega}$. We next show that if $Q$ is a finite quandle (not necessarily a conjugation quandle), then $Q\in\mathcal{R}_{\omega}$ if and only if  $Q\in\mathcal{LR}_{\omega}$. 

A subset $H$ of a group $G$ is called {\it normal} if $x^{-1}Hx=H$ for all $x\in G$. The {\it Fitting subgroup} of $G$ is the subgroup generated by all nilpotent normal subgroups of $G$. The following result follows from \cite[Theorem 2.8]{Abd}.
\begin{theorem}\label{group_gen_by_n_engel}
	Let $G$ be a group that satisfies the maximal condition on
	abelian subgroups. Then every normal $n$-Engel subset of $G$  is contained in the Fitting subgroup of $G$. 
\end{theorem}
This theorem implies the following statement.
\begin{theorem}
	\label{finite M=R}
	Let $Q$ be a finite quandle. Then $Q$ is reductive if and only if $Q$ is locally reductive. \end{theorem}
	
\begin{proof}If $Q$ is reductive, then it is obviously locally reductive (since $\mathcal{R}_{\omega}\subset\mathcal{LR}_{\omega}$). Suppose that $Q$ is a finite $n$-locally reductive quandle. According to Proposition \ref{innred}, the group $\lmlt{(Q)}$ is generated by the normal $n$-Engel subset $\setof{L_a}{a\in Q}$, and since $\lmlt{(Q)}$ is finite, it satisfies the maximal condition on abelian subgroups. From Theorem \ref{group_gen_by_n_engel}, it follows that ${\rm Inn}(Q)$ coincides with its Fitting subgroup. Since ${\rm Inn}(Q)$ is finite, its Fitting subgroup is nilpotent, and therefore ${\rm Inn}(Q)$ is nilpotent. Thus, by Proposition~\ref{strong red and inn}, $Q$ is reductive.
\end{proof}
In particular, by Corollary \ref{trivialreductive}, if a finite locally reductive quandle is either faithful or connected, it must be $T_1$. 

Similarly, for medial quandles, reductivity and local reductivity are equivalent  \cite{JPSZ}.  Therefore, only $T_1$ is medial, connected, and locally reductive. Hence, all the medial connected quotients of a locally reductive quandle are trivial.  From this observation, we can provide the following proposition about the transvection group of connected locally reductive quandles.

\begin{proposition}\label{perfect}
	Let $Q$ be a connected locally reductive quandle. Then $Q$ has no finite quotients and $\dis(Q)$ is a perfect group, i.e., $\dis(Q)=[\dis(Q),\dis(Q)]$.
\end{proposition}
\begin{proof} Finite locally reductive quandles are not connected, so $Q$ has no finite quotients.

According to \cite{BiaBon}, the quotient of $Q$ with respect to the congruence $\alpha=\mathcal{O}_{[\dis(Q),\dis(Q)]}$  is a connected medial locally reductive quandle. It follows from Theorem~\ref{finite M=R} that $Q/\alpha = T_1$. Therefore, $\dis_\alpha=\dis(Q)\leq [\dis(Q),\dis(Q)],$ so  $\dis(Q)=[\dis(Q),\dis(Q)]$.

\end{proof}

\section{Orbit series in quandles}\label{orbitserrrr}

\subsection{Definitions and examples} Let $Q=Q_0$ be a quandle and $x_0 \in Q_0$. Let $Q_1={\rm Orb}(x_0,Q)$ and choose an element $x_1 \in Q_1$. Continuing in the same way (choosing $x_{i+1} \in Q_{i+1} = {\rm Orb}(x_{i}, Q_i)$, letting $Q_{i+2}={\rm Orb}(x_{i+1}, Q_{i+1})$, and so on), we obtain a sequence of subquandles:
\begin{equation}\label{os}
Q=Q_0\geq Q_1\geq Q_2\geq Q_3\geq \dots
\end{equation}

Such a series of subquandles is called the {\it orbit series} defined by the elements $\setof{x_i}{i\in \mathbb{N}}$. Note that the same orbit series can be obtained by different sequences of elements.
If an orbit series arises from a constant sequence of elements, namely all $x_i=x$ for some $x\in Q$, then we say that this orbit series is the {\it principal orbit series} of $x$. A quandle $Q$ has a unique orbit series if and only if it is connected. We will use the following terminology:
\begin{itemize}
	\item A quandle $Q$ satisfies the {\it descending orbit series condition} if each orbit series as in \eqref{os} stabilizes, i.e., there exists $k\in \mathbb{N}$ such that  $Q_k=Q_{k+1}$. We denote by $\OC$ the class of quandles which satisfy this condition. 
	\item A quandle $Q$ satisfies the {\it $n$-bounded descending orbit series condition} if each orbit series as in \eqref{os} stabilizes in $n$ or fewer steps, i.e.,
	there exists $k\leq n$ such that  $Q_k=Q_{k+1}$. We denote the class of quandles which satisfy this condition by $\OC_n$ and let $\OC_{\omega}=\bigcup_{n=1}^{\infty}\OC_n$.
	\item A quandle $Q$ satisfies the {\it trivializing orbit series condition} if each orbit series as in \eqref{os} stabilizes on $T_1$, i.e., there exists $k\in \mathbb{N}$ such that $Q_k = T_1$. The class of all quandles which satisfy this condition will be denoted by $\dOC$.
	\item 
	A quandle $Q$ satisfies the {\it $n$-bounded trivializing orbit series condition} if for each orbit series as in \eqref{os}, there exists $k\leq n$ such that  $Q_k = T_1$. The class of quandles which satisfy this condition is called $\dOC_n$  and  $\dOC_{\omega}=\bigcup_{n=1}^{\infty}\dOC_n$.
\end{itemize}
It is clear that every finite quandle satisfies the orbit series condition. The following inclusions follow from the definitions:
\begin{equation}\label{conditionsinclusions}
\begin{matrix}
\dOC_1&\subset&\dOC_2&\subset&\dots&\subset&\dOC_{\omega}&\subset&\dOC\\
\cap&~&\cap&~&~&~&\cap&~&\cap\\
\OC_1&\subset&\OC_2&\subset&\dots&\subset&\OC_{\omega}&\subset&\OC
\end{matrix}
\end{equation}
The following examples show that all inclusions in \eqref{conditionsinclusions} are strict.

\begin{example}\label{dihgood}  The dihedral quandle $D_{2n}$ is a union of two orbits, each of which is isomorphic to $D_n$  \cite[Proposition 3.2]{BarNas}.  Therefore, for $1\leq k\leq n$, the $k$-th member of every orbit series of $D_{2^n}$ is isomorphic to $D_{2^{n-k}}$.  This means that $D_{2^n}$ belongs to $\dOC_{n}\setminus\dOC_{n-1}$ and to $\OC_{n}\setminus\OC_{n-1}$. So, the inclusions $\OC_{n-1} \subset \OC_{n}$, $\OC_n\subset \OC_{\omega}$, $\dOC_{n-1}\subset \dOC_{n}$, and $\dOC_n\subset \dOC_{\omega}$ are strict.
\end{example}

\begin{example}\label{omegastrict}
There is a natural arbitrary direct sum operation on quandles.  Namely, given a set $\mathcal{S}$ of quandles, relabeled if necessary so that their underlying sets are disjoint, then $\bigsqcup \mathcal{S}$ is the quandle on the union of the underlying sets with operation given by:
$$x \rhd y =  \left\{ \begin{array}{ll}
x \rhd_Q y & \textrm{ if }  \exists Q \in \mathcal{S} \textrm{ such that } x, y \in Q \\
y  & \textrm{ otherwise } 
\end{array} \right. $$
The disjoint union  $Q=\bigsqcup_{n\geq 1} D_{2^n}$ belongs to $\dOC\setminus\dOC_{\omega}$ and so  the inclusions $\OC_{\omega}\subset \OC$ and $\dOC_{\omega}\subset \dOC$ are strict.
\end{example}
\begin{example}
Every connected quandle belongs to $\OC_1$. If $Q$ is a connected non-trivial quandle, then it has a unique orbit series $Q=Q=\dots$, which of course does not stabilize on $T_1$. Therefore, the inclusions $\dOC_n\subset \OC_n$, $\dOC_{\omega}\subset \OC_{\omega}$, and $\dOC\subset \OC$ are strict.
\end{example}

The orbit conditions defined above can be captured by the properties of the following tree.  Let $Q$ be a quandle. Denote by $V(Q)$ the following inductively defined set of subquandles of $Q$;
\begin{enumerate}
	\item[(i)] $Q\in V(Q)$, and
	\item[(ii)] if $P\in V(Q)$ and $R$ is an orbit in $P$, then $R\in V(Q)$.
\end{enumerate}

The graph $\Gamma(Q)$ has vertices $V(Q)$ and edges $\setof{(P,R)}{P,R\in V(Q),\,  R \neq P \,\text{ is an orbit in } P}$. It is clear that $\Gamma(Q)$ is a tree with root $Q$ and that every leaf $P$ is a connected subquandle of $Q$: If $P$ is not connected, then it has an orbit $R$ which does not coincide with $P$ and then we have an edge $(P,R)$.  If the tree has finite depth, then the leaves are the maximal connected subquandles of $Q$. We call $\Gamma(Q)$ the {\it orbit tree} of the quandle $Q$.

\begin{example}
The following are examples of orbit trees.
\begin{itemize}

\item If $Q$ is a connected quandle, then the graph $\Gamma(Q)$ has only one vertex $Q$.

\item For the trivial quandle $T_n=\{x_1,\dots,x_n\}$, we have: 
	$$ \xymatrixrowsep{0.2in}
	\xymatrixcolsep{0.1in}
	
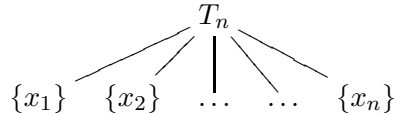
\xymatrix{
 &  & T_n  \ar@{-}[d]\ar@{-}[drr]\ar@{-}[dll] \ar@{-}[dr]\ar@{-}[dl]  &  & \\
\{x_1\} & \{x_2\} & \ldots  & \ldots & \{x_n\}
}
	$$

\item For the dihedral quandle  $D_4$, we have:   
	$$ \xymatrixrowsep{0.2in}
	\xymatrixcolsep{0.00001in}
	
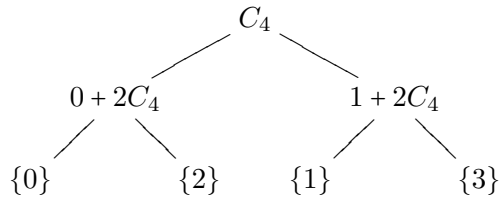
\xymatrix{
& & & C_4 \ar@{-}[drr]\ar@{-}[dll] \\
& 0+2C_4\ar@{-}[dr]\ar@{-}[dl]  & & & &1+2C_4\ar@{-}[dr]\ar@{-}[dl]  \\
\{0\}& & \{2\} & & \{1\}& &\{3\}  \\
}
	$$

\end{itemize}

\end{example}

The following statement recasts the definitions of the classes $\OC,\OC_n,\dOC,$ and $\dOC_n$ in terms of the orbit tree. 
\begin{proposition}\label{intuition} Let $Q$ be a quandle. Then
\begin{enumerate}
	\item[(i)] $Q\in \OC$ if and only if each branch of $\Gamma(Q)$ has finite depth.
	\item[(ii)] $Q\in \OC_n$ if and only if each branch of $\Gamma(Q)$ has depth which is less then or equal to $n$.
	\item[(iii)] $Q\in \dOC$ if and only if $Q \in \OC$ and every leaf is $T_1$.
	\item[(iv)] $Q\in \dOC_n$ if and only if $Q \in \OC_n$ and every leaf is $T_1$.
\end{enumerate}

\end{proposition}

Intuitively, the families $\dOC$ and $\dOC_n$  are assembled from copies of $T_1$ in a finitary way.  Similarly, quandles from the families $\OC$ and $\OC_n$ are assembled from connected quandles.

We now characterize principal orbit series.
\begin{proposition}\label{easieros2} Let $Q$ be a quandle and $U=\setof{Q_i}{i\in \mathbb{N}}$ be an orbit series. Then $U$ is the principal orbit series of $x$ if and only if  $x\in \bigcap_{i\in \mathbb{N}} Q_i$.
\end{proposition}
\begin{proof}
	Clearly, if $U$ is the principal orbit series of $x$, then $x\in \bigcap_{i\in\mathbb{N}} Q_i$. 	Conversely, let $x$ be an arbitrary element from $\bigcap Q_i$. Since $Q_{i+1}={\rm Orb}(x_i,Q_i)$ and $x\in Q_{i+1}$, the orbit of $x_i$ in $Q_i$ and the orbit of $x$ in $Q_i$ must coincide. So, we have ${\rm Orb}(x,Q_i)={\rm Orb}(x_i,Q_i)=Q_{i+1}$.
\end{proof}

If $Q$ satisfies the descending orbit series condition, then each orbit series in $Q$ stabilizes. Therefore, each orbit series has non-trivial intersection and Proposition~\ref{easieros2} has the following corollary.
\begin{corollary}\label{easieros} Let $Q\in \OC$. Then every orbit series in $Q$ is principal.
\end{corollary}
Every finite quandle satisfies the descending orbit series condition, leading to:
\begin{corollary} Let $Q$ be a finite quandle. Then every orbit series in $Q$ is principal.
\end{corollary}
In some infinite quandles, there are orbit series that have empty intersection. So there are quandles with non-principal orbit series; we give a concrete example.
\begin{example}\label{notOC} Let $Q={\rm Alex}(\Z,-)$ be the Takasaki quandle on $\mathbb{Z}$, also known as the infinite dihedral quandle. The operation in $Q$ is given by $a\rhd b=2a-b$ for integers $a$ and $b$. If $a\in Q$, then the orbit of $a$ has the form 
	$${\rm Orb}(a,Q)=\{2b-a~|~b\in \mathbb{Z}\}=\{c~|~c~\text{and}~a~\text{have the same parity}\}.$$
	Therefore, $Q$ has two orbits: ${\rm Orb}(0,Q)=\{\text{even numbers}\}$ and ${\rm Orb}(1,Q)=\{\text{odd numbers}\}$. From the equations $$(2a)\rhd (2b)=2(2a-b) \quad \textrm{and} \quad (2a+1)\rhd (2b+1)=(2(2a-b)+1)$$
	it follows that each of these two orbits is isomorphic to $Q$. By the same argument, the orbits within these orbits will again be isomorphic to $Q$. So if we construct an orbit series $Q=Q_0\geq Q_1\geq Q_2\geq \dots$, then each of the quandles $Q_i$ will be isomorphic to $Q$. In fact, by direct computation, the orbit tree of $Q$ is as follows:  
	\begin{footnotesize}
	
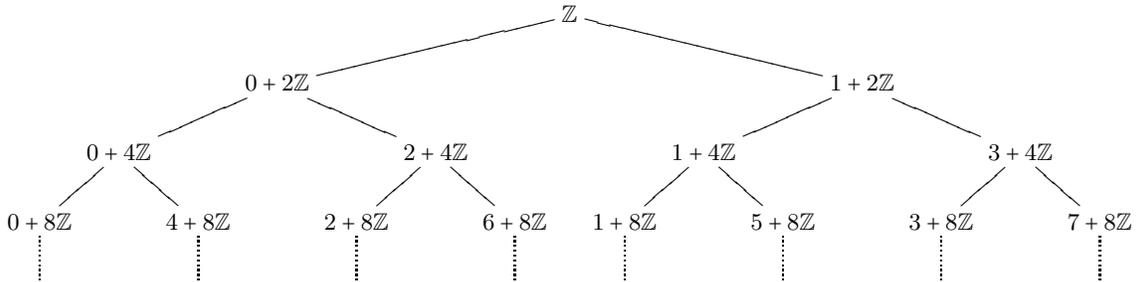
	$$ \xymatrixrowsep{0.2in}
	\xymatrixcolsep{0.00001in}
	\xymatrix{
& & & & & & & \Z  \ar@{-}[drrrr]\ar@{-}[dllll]& & & & & & & &\\
& & & 0+2\Z \ar@{-}[drr]\ar@{-}[dll] & & & & & & & & 1+2\Z \ar@{-}[drr]\ar@{-}[dll]  & & &\\
& 0+4\Z\ar@{-}[dr]\ar@{-}[dl]  & & & &2+4\Z\ar@{-}[dr]\ar@{-}[dl]  && & & 1+4\Z\ar@{-}[dr]\ar@{-}[dl] & & &&  3+4\Z\ar@{-}[dr]\ar@{-}[dl] & \\
0+8\Z \ar@{.}[d] & & 4+8\Z \ar@{.}[d] & & 2+8\Z \ar@{.}[d]& &6+8\Z \ar@{.}[d]& &1+8\Z \ar@{.}[d]& &5+8\Z \ar@{.}[d]& &3+8\Z \ar@{.}[d]& &7+8\Z\ar@{.}[d] \\
 & & & & & & & & & & & & & & 
}
	$$
	\end{footnotesize}

The orbit series of $Q$ are thus in one-to-one correspondence with the $2$-adic integers $\Z_2$.  Recall that $\Z_2$ is the subring of $\prod_{n\geq 1}\Z/2^n\Z$ given by the elements $(a_n)_{n\geq 1}$ such that $a_{n}=\rho_{n+1}(a_{n+1})$ for every $n\geq 1$, where $\rho_{n+1}$ is the canonical map making the following diagram commute
$$\xymatrix{ \Z\ar[r]\ar[d] & \Z/2^{n}\Z\\
\Z/2^{n+1}\Z\ar[ru]^{\rho_{n+1}} & 
}$$
The integers embed into $\Z_2$ as the eventually constant sequences. In particular, $x\in\Z$ is in the intersection of the orbit series $\setof{a_n+2^n \Z}{n\in \mathbb{N}}$ if and only if, in $\Z_2$, $x=(a_1,a_2,\ldots)$. So choose any $2$-adic integer not in $\Z$, for instance $1/3 = (1, 3, 3, 11, 11, 43, 43,\ldots)$. The intersection of the orbit series corresponding to $1/3$ is empty, and so this orbit series cannot be principal.
\end{example}

\section{Orbit series and the structure of quandles}\label{section about properties}

\subsection{Closure properties}\label{subquandles and so on} Here we study the behavior of the families $\OC$, $\OC_n$, $\dOC$, and $\dOC_n$ under taking subquandles, homomorphic images, direct products, and extensions.
\begin{lemma}\label{series of factor}Let $Q$ be a quandle and $\alpha\in {\rm Con}(Q)$. Let $\setof{Q_n}{n\in\mathbb{N}}$ and $\setof{(Q/\alpha)_n}{n\in\mathbb{N}}$ be the orbit series determined by the sequences $\setof{a_n}{n\in\mathbb{N}}$ and  $\setof{[a_n]_\alpha}{n\in\mathbb{N}}$, respectively. Then for all $n, (Q/\alpha)_n= Q_n/\alpha$.
\end{lemma}
\begin{proof} We induct on $n$ to show $\pi_\alpha(\lmlt{(Q_n)})=\lmlt{((Q/\alpha)_n)}$. The base case follows from the definition of $\pi_\alpha$.
Now, assume that $\pi_\alpha(\lmlt{(Q_n)})=\lmlt{((Q/\alpha)_n)}$. Note that as an arbitrary automorphism $h$ varies over all of  $\textrm{Inn}(Q_n)$, the map $L_{h(a_n)}$ runs over all the generators of $\textrm{Inn}(Q_{n+1})$ and $\pi_{\alpha}(h)$ runs over all of $\textrm{Inn}((Q/\alpha)_n)$. For any such $h$, we have:
$$\pi_{\alpha}(L_{h(a_n)}) = L_{[h(a_n)]_{\alpha}} = L_{\pi_{\alpha}(h)([a_n]_{\alpha})} \in \textrm{Inn}((Q/\alpha)_{n+1}),$$
where the second equality is by equation \ref{piprop}.  Further, the right-hand side of the second equality hits every generator of $\textrm{Inn}(Q/\alpha)_{n+1}$.  Thus $\lmlt((Q/\alpha)_{n+1})=\pi_\alpha(\lmlt(Q_{n+1}))$. 
	
	Using this, we see that
	\begin{align}
	\notag	(Q/\alpha)_{n}&=\setof{h([a_{n}])}{h\in \lmlt{((Q/\alpha)_{n-1})}}\\
	\notag	&=\setof{\pi_\alpha(h)([a_{n}])}{h\in \lmlt{(Q_{n-1})}}\\
	\notag	&=\setof{[h(a_{n})]}{h\in \lmlt{(Q_{n-1})}}=\setof{[b]_\alpha}{b\in Q_n}
	\end{align}
which completes the proof.
\end{proof}
\begin{proposition}\label{Nbeh} The classes $\dOC_n$, $\dOC_{\omega}$, $\dOC$, $\OC_n$, $\OC_{\omega}$, and $\OC$ are closed under taking homomorphic images and finite direct products. The classes $\dOC_n$, $\dOC_{\omega},$ and $\dOC$ are also closed under taking subquandles.
\end{proposition}
\begin{proof}
	We provide the details for the case of $\dOC_n$; other cases are proved analogously.
	Let $Q\in \dOC_n$ and $\alpha\in {\rm Con}(Q)$. We will show that $Q/\alpha\in \dOC_n$. According to Lemma \ref{series of factor}, the orbit series in $Q/\alpha$ are images of orbits series of $Q$ under the canonical map. Hence, if $Q_n=T_1$, then $(Q/\alpha)_n =T_1$ also.
	
	To see that these properties are stable under taking finite direct products, it is enough to note that the elements of an orbit series in the direct product are given by tuples of elements of orbit series in the quandles making up the direct product. 
	
	Let $Q\in \dOC_n$ and $P$ be a subquandle of $Q$. Let $\setof{P_i}{i \in \mathbb{N}}$ be an orbit series defined by $x_i \in P_i$ and $\setof{Q_i}{i \in \mathbb{N}}$ be the orbit series defined by the same $x_i$ in $Q$. By inducting on $i$, it is easy to see that $P_i\leq Q_i$. Since  $Q\in \dOC_n$, $Q_n = T_1$, forcing $P_n = T_1$ as well. Therefore $P\in\dOC_n$.
\end{proof}

\begin{example} The Takasaki quandle  ${\rm Alex}(\mathbb{Q},-)$ is connected and therefore it belongs to the families $\OC_n$, $\OC_{\omega}$, and $\OC$. However, ${\rm Alex}(\mathbb{Q},-)$ has a subquandle ${\rm Alex}(\Z,-)$ that does not belong to $\OC$ (see Example~\ref{notOC}).  This example shows that the classes $\OC_n$, $\OC_{\omega}$, and $\OC$ are not closed under taking subquandles.
\end{example}

\begin{example} Recall $D_{2^n} \in\dOC_{n}$ (see Example~\ref{dihgood}). Therefore, $D_{2^n}$ belongs to the families $\dOC_{\omega}, \dOC, \OC_{\omega},$ and $\OC$. However, the quandle $\prod_{n\in \mathbb{N}} D_{2^n}$ does not belong to $\OC$ (and, therefore, does not belong to $\dOC_{\omega}, \dOC, \OC_{\omega}$, or $\OC$).  This example shows that the families $\dOC_{\omega}$, $\dOC$, $\OC_{\omega}$, and $\OC$ are not closed under taking arbitrary direct products. 
\end{example}

\begin{lemma}\label{extenstionssss} Let $Q$ be a quandle and $\alpha\in {\rm Con}(Q)$. If $Q/\alpha\in \dOC_n$ and $[a]_{\alpha}\in \dOC_m$ for all $a\in Q$, then $Q\in \dOC_{n+m}$.
\end{lemma}

\begin{proof} If $Q/\alpha \in \dOC_n$, then the $n$-th element of every orbit series of $Q$ is contained in one equivalence class of $\alpha$, i.e., $Q_n\subseteq [a]_{\alpha}$ for some $a\in Q$. If the subquandle $[a]_\alpha$ belongs to $\dOC_m$, then $Q_{n+m}=(Q_n)_m = T_1$, so $Q\in \dOC_{n+m}$. 
\end{proof}

\begin{corollary}\label{doc is closed under extensions}
The classes $\dOC_{\omega}$ and $\dOC$ are closed under taking extensions. 
\end{corollary}
Thus, the only unsettled closure condition is:
\begin{question}
Are $\OC$ and $\OC_{\omega}$ closed under taking extensions?
\end{question} It turns out that there exist extensions of infinite trivial quandles by $D_3$ that do not decompose as direct products. Therefore, a positive answer to this question seems unlikely, although a specific counterexample has not been identified.

Note also that the classes $\dOC_{\omega}$ and $\dOC$ contain the trivial quandles and the classes of $\lambda_Q$ are trivial.  
\begin{corollary}\label{reduction to the conjugation case}
	A quandle $Q$ belongs to one of the classes $\dOC_{\omega}$ or $\dOC$ if and only if $Q/\lambda_Q$ belongs to $\dOC_{\omega}$ or $\dOC,$ respectively.
\end{corollary}
Given that $Q/\lambda_Q$ is a subquandle of ${\rm Conj}_{-1}({\rm Inn}(Q))$, see Section \ref{conghom}, this corollary means it suffices to investigate membership in these classes just for conjugation quandles.

\subsection{Orbit series conditions and connected subquandles}\label{are there connected subquandles?}
Let $\nCS$ be the class of quandles that have no non-trivial connected subquandles. These quandles are closely related to the families defined by the orbit series conditions.
\begin{proposition}
	The class $\nCS$ is closed under taking subquandles, extensions, and direct products. 
\end{proposition}

\begin{proof}
	Clearly every subquandle of a quandle with no non-trivial connected subquandles has no non-trivial connected subquandles. For an arbitrary index set $I$, let $\setof{Q_i}{i\in I}$ be a family of quandles each of which has no connected subquandles.  If $M$ is a connected subquandle of  $\prod_{i\in I} Q_i$, then its projection onto $Q_i$ is trivial for every $i\in I$, and therefore $M$ is trivial.

Let $N$ be a connected subquandle of a quandle $Q$. If $Q/\alpha$ has no non-trivial connected subquandles, then $N\subseteq [a]_\alpha$ for some $a\in Q$. If $[a]_\alpha$ has no non-trivial connected subquandles, then $N$ is trivial. \end{proof}

Note that the class $\nCS$ is not closed under taking homomorphic images of infinite quandles.  Consider $Q = {\rm Alex}(\Z,-)$.  On the one hand, $Q/\!\equiv_3 \;  \, = D_3$ is connected.  On the other hand, if $S$ is a subquandle of $Q$, we may assume $0 \in S$ by homogeneity of $Q$.  If $S$ is non-trivial, choose $0 \neq a \in S$ of minimal absolute value.  Using the quandle operation, it is not difficult to see that $a\Z \subseteq S$.  Now suppose $qa + r \in S$ for $q \in \mathbb{Z}$ and $a > r \in \mathbb{N}$.  If $q$ is even, then $(q/2)a \rhd (qa + r) = -r \in S$.  If $q$ is odd, then $((q + 1)/2)a \rhd (qa + r) = a - r \in S$.  In either case, $S$ would contain an element of absolute value smaller than $a$.  Hence, $S = a\mathbb{Z}$.  But now the orbit of $2a \in S$ is $2a \Z$, so $S$ is not connected.

The example above also shows that the classes $\nCS$ and $\dOC$ are distinct. The following statement describes $\dOC$ in terms of its connected subquandles.
\begin{proposition}\label{finite descending quandles}
	$\dOC=\OC\cap \nCS$.
\end{proposition}
\begin{proof}
	$(\subseteq)$ Suppose that $P$ is a non-trivial connected subquandle of $Q$ and let $x\in P$. Then $P$ is contained in the intersection of the principal orbit series of $x$.  Therefore, this orbit series cannot stabilize on $T_1$.

	$(\supseteq)$  Assume that $Q$ does not belong to $\dOC$, i.e., there exists an orbit series $Q_0\geq Q_1\geq\dots$ that does not stabilize on $T_1$. Since $Q\in\OC$, this orbit series must stabilize, i.e., there exists a number $k$ such that $Q_k=Q_{k+1}$. But by definition, $Q_{k+1}$ is an orbit in $Q_k$, and therefore $Q_k$ is a non-trivial connected subquandle of $Q$.
\end{proof}
Since every finite quandle belongs to $\OC$, Proposition~\ref{finite descending quandles} has the following corollary.
\begin{corollary}\label{ncs for finite}
	Let $Q$ be a finite quandle. Then $Q\in \dOC_{\omega}$ if and only if $Q\in \nCS.$
\end{corollary}
Note that the statement of Corollary ~\ref{ncs for finite} does not necessarily hold for infinite quandles. For example, we just noted that ${\rm Alex}(\Z,-) \in \nCS$, and saw in Example \ref{notOC} that it does not satisfy the descending orbit series condition.  On the other hand, Corollary \ref{ncs for finite} does mean that $\nCS$ is closed under homomorphic images of finite quandles.

\begin{remark}
	Each element $(a,b) \in Q={\rm Alex}(\Z,-) \times {\rm Alex}(\mathbb{Q},-)$ is contained in the connected subquandle $\{a\}\times {\rm Alex}(\mathbb{Q},-)$. But ${\rm Alex}(\Z,-)$ is a homomorphic image of $Q$ that is not in $\OC$, and therefore $Q$ does not belong to $\OC$. So, not even being a union of  connected subquandles is a sufficient condition for a quandle to be in $\OC$.
\end{remark}

\subsection{Descending series conditions, reductive and locally reductive quandles}\label{orbmultred}
In this subsection we study connections between the families  $\dOC$ and  $\dOC_n$ and the families of reductive, locally reductive, and solvable quandles. 
\begin{proposition}\label{MnindOCn}$\mathcal{R}_n\subseteq \dOC_n\subseteq \mathcal{LR}_n$ for every $n\in \mathbb{N}$.
\end{proposition}
\begin{proof}
	It is easy to check, by induction, that the $n$-th element of every orbit series of a quandle $Q$ is contained in a single equivalence class of $\mathcal{O}^n_Q$. According to Proposition \ref{strong red and inn}, if $Q$ is $n$-reductive, then $\mathcal{O}^n_Q=0_Q$ and therefore the $n$-th element of every orbit series of $Q$ is trivial. Hence $Q\in \dOC_n$. 
	
	If $Q\in\dOC_1$, then it is $1$-locally reductive. Assume that $Q\in\dOC_n$. Then the orbits belong to $\dOC_{n-1}$ and so, by induction, they are $(n-1)$-locally reductive. By Lemma \ref{reductivity by component}, $Q$ is $n$-locally reductive.
\end{proof}
\begin{corollary} Let $G$ be a nilpotent group of nilpotency class $n$. Then ${\rm Conj}(G)\in \dOC_{n}$.
\end{corollary}
\begin{proof}In this case ${\rm Inn}({\rm Conj}(G))=G/Z(G)$ has nilpotency class $n-1$, and therefore, by Proposition~\ref{strong red and inn}, {\rm Conj}(G) belongs to $\mathcal{R}_n$. By Proposition~\ref{MnindOCn}, it belongs to  $\dOC_{n}$.
\end{proof}

The following example shows that the inclusions in Proposition \ref{MnindOCn} are strict (except for the case $n=1$ of trivial quandles).
\begin{example}\label{example1}
	Let $Q$ be the quandle with the following operation table: 
	\bigskip
	\begin{center}$Q= \left[
		\begin{tabular}{ c c c c  c c c c || c c c c || c c c c}
			1&2&4&3&5&6&7&8   &11&12&9&10&15&16&13&14 \\
1&2&4&3&5&6&7&8 &11&12&9&10&15&16&13&14 \\
2&1&3&4&5&6&7&8  &10&9&12&11&14&13&16&15 \\
2&1&3&4&5&6&7&8&10&9&12&11&14&13&16&15 \\
1&2&3&4&5&6&8&7&11&12&9&10&14&13&16&15 \\
1&2&3&4&5&6&8&7&11&12&9&10&14&13&16&15 \\
1&2&3&4&6&5&7&8&10&9&12&11&15&16&13&14 \\
1&2&3&4&6&5&7&8&10&9&12&11&15&16&13&14 \\ \hline
\hline
5&6&7&8&1&2&3&4&9&10&11&12&13&15&14&16 \\
6&5&7&8&2&1&3&4&9&10&11&12&16&14&15&13 \\
5&6&8&7&1&2&4&3&9&10&11&12&16&14&15&13 \\
6&5&8&7&2&1&4&3&9&10&11&12&13&15&14&16 \\ \hline \hline
7&8&5&6&3&4&1&2&9&11&10&12&13&14&15&16 \\
8&7&5&6&3&4&2&1&12&10&11&9&13&14&15&16 \\
7&8&6&5&4&3&1&2&12&10&11&9&13&14&15&16 \\
8&7&6&5&4&3&2&1&9&11&10&12&13&14&15&16 \\
		\end{tabular}\, \right]$
	\end{center}\medskip
The orbits of $Q$ are (evidently)  $\{ 1, \ldots, 8\}, \{ 9, \ldots, 12 \}$ and $\{ 13, \ldots, 16\}$. The first orbit is the direct sum of two copies of $D_4$, which is not trivial, and so $Q \in \dOC_3$, but $Q \nin \dOC_2$. On the other hand, $Q$ is $2$-locally reductive, $5$-reductive, and, critically, not $3$-reductive.  This example was constructed so that the action of the other two trivial orbits on $D_4 \sqcup D_4$ does not disturb its $2$-local reductivity.   
\end{example}

The next two examples establish the strictness of other inclusions.  The first shows that $\mathcal{R_{\omega}}\neq \dOC_{\omega}$ and the second illustrates that $\mathcal{LR}_{\omega} \not\subseteq \dOC$.
\begin{example}\label{M neq dOC}Let $G$ be a $3$-Engel group that is not nilpotent (as given in Example 18.3.1 of \cite{KarMer}) and let $Q={\rm Conj}(G)$. Since $G$ is $3$-Engel, it follows that $\langle x^G\rangle = {\rm Orb}(x,Q)$ is nilpotent of length $2$ for every $x\in G$ \cite{KapKap}. Hence $\rm{Inn}(\textrm{Orb}(x,Q))$ is abelian. By Proposition~\ref{strong red and inn}, this means that each orbit of $Q$ belongs to $\mathcal{R}_2\subseteq \dOC_2$, and so $Q \in \dOC_3$ by the definition of orbit series. On the other hand, $G$ is not nilpotent, so $Q \nin \mathcal{R_{\omega}}$ by Proposition~\ref{strong red and inn}. 
\end{example}

\begin{example} \label{Wonderful example of Marco} Let $B=B(2,2^k)$  be the free Burnside group of rank $2$ and exponent $2^k$ with $k$ large enough (so that $B$ is infinite). By the solution of the restricted Burnside problem, the group $B$ has only finitely many subgroups of finite index $\setof{B_i}{1\leq i\leq m}$. So $B^0=\bigcap_{i=1}^m B_i$ is a finite-index subgroup of $B$ \cite{Zelmanov}. The subgroup $B^0$ has no finite-index subgroup and it has exponent dividing $2^k$. Moreover, $B^0$ is finitely generated since it is a finite-index subgroup of a finitely generated group. Therefore $B^0$ has an infinite simple quotient $G$ that also has exponent dividing $2^k$.
	
Let $g\in G$ be an element of order two and $\widehat{g}$ be the inner automorphism induced by $g$. Let $Q={\rm Alex}(G,\widehat{g})$ be the Alexander quandle of the group $G$ with
respect to the automorphism $\widehat{g}$. The operation in $Q$ is given by $x\rhd y=\widehat{g}(yx^{-1})x$. Directly from the definition of the operation, ${\rm Orb}(1,Q) = \langle\{\widehat{g}(x)^{-1}x~:~x\in G\}\rangle$. Since the subgroup $\langle\{\widehat{g}(x)^{-1}x~:~x\in G\}\rangle$ is normal in $G$ (see, for example, \cite[Proof of Proposition~1]{BarNasNes} or \cite{GonNas}), and $G$ is simple, it follows that $\langle\{\widehat{g}(x)^{-1}x~:~x\in G\}\rangle=G$. Hence, ${\rm Orb}(1,Q)=G=Q$ and so $Q$ is connected. In particular, $Q$ does not belong to $\dOC$ (nor $\dOC_{\omega}$).

Now we show that $Q \in \mathcal{LR}_{k + 1}$. Define a function $t: Q \rightarrow Q$ by $t(x) = x \rhd 1 = \widehat{g}(x)^{-1}x$.  Since $g$ has order two, for an arbitrary $x\in G$ we have the equalities
\begin{equation*}\label{Wow Marco}\widehat{g}(t(x))=\widehat{g}(\widehat{g}(x)^{-1}x)=\widehat{g}^2(x)^{-1}\widehat{g}(x)=x^{-1}\widehat{g}(x)=t(x)^{-1}.
\end{equation*}
We show that $t^n(x)=t(x)^{2^{n-1}}$ by induction on $n$. The basis of induction is trivial. Using the above equation for $\widehat{g}(t(x))$ and the induction hypothesis we have 
	$$t^n(x)=\widehat{g}\left(t^{n-1}(x)\right)^{-1}t^{n-1}(x)=\widehat{g}\left(t(x)^{2^{n-2}}\right)^{-1}t(x)^{2^{n-2}}=t(x)^{2^{n-2}}t(x)^{2^{n-2}}=t(x)^{2^{n-1}}$$
which completes the induction. Translating this to the operation in $Q$, and considering the particular case $n = k+ 1$, we have: 
\begin{equation*}\label{almost reductive}(\dots((a \rhd \underbrace{1)\rhd 1)\rhd \ldots ) \rhd 1}_{k+1} =t^{k+1}(a)= t(a)^{2^k}= 1
\end{equation*}
for all $a\in Q$. Since $Q$ is connected, apply the inner automorphism taking $1$ to $b$ to see that
$$(\dots((a \rhd \underbrace{b)\rhd b)\rhd \ldots ) \rhd b}_{k+1} =b
$$
holds for all $a,b\in Q$, i.e.,  $Q$ is $k+1$-locally reductive. 
\end{example}

Proposition \ref{MnindOCn} together with Examples \ref{example1} through \ref{Wonderful example of Marco} and the system of inclusions (\ref{MRinclusions}) imply the following theorem.
\begin{theorem}\label{ourinclusions}The following system of inclusions holds 
$$
	\begin{matrix}\mathcal{R}_n&\subset&\dOC_n&  \subset &  \mathcal{LR}_n\\
	\cap&&\cap& & \cap\\
	\mathcal{R}_{\omega}&\subset&\dOC_{\omega}& \subset& \mathcal{LR}_{\omega}
	\end{matrix}\, 
$$
	and all inclusions in this system are strict.
\end{theorem}
By Proposition \ref{finite M=R}, finite quandles are reductive if and only if they are locally reductive. Hence,   Theorem~\ref{ourinclusions} has the following corollary.
\begin{corollary}\label{equivalences_for_finite}
	Let $Q$ be a finite quandle. Then $Q\in \mathcal{R}_{\omega} \Leftrightarrow Q\in \dOC\Leftrightarrow Q\in \mathcal{LR}_{\omega}$. 

\end{corollary}
Similarly, if $Q$ is medial, then it is $n$-locally reductive if and only if it is $n$-reductive \cite{JPSZ}.  
\begin{corollary}\label{equivalences for medial}
	Let $Q$ be a medial quandle and $n\in \mathbb{N}$. Then $Q\in \mathcal{R}_n\Leftrightarrow Q\in \dOC_n\Leftrightarrow Q\in \mathcal{LR}_n$. 
\end{corollary}
This corollary lets us see that the locally reductive order of an extension established in Lemma \ref{R is cloded for} is the best possible:  The dihedral quandle $D_{2^{n+m}}$ is medial, hence in $\mathcal{LR}_{n+m} \setminus \mathcal{LR}_{n+m -1}$.  However, it is not difficult to see that $D_{2^{n+m}}$ has $D_{2^n}$ as a quotient, with each equivalence class isomorphic to $D_{2^m}$.

Abelian quandles are medial, so Corollary~\ref{equivalences for medial} can be applied to abelian quandles.  The following statement partially extends  Corollary~\ref{equivalences for medial} from abelian quandles to solvable quandles. 
\begin{proposition}\label{solvable doc iff R}
	Let $Q$ be a solvable quandle. Then $Q\in \dOC_{\omega}$ if and only if $Q\in \mathcal{LR}_{\omega}$. In particular, if $Q$ is $k$-locally reductive and solvable of length $n$, then $Q\in \dOC_{nk}$.
\end{proposition}
\begin{proof}
	The claim is true for abelian quandles.  Assume, by induction, that the claim is true for quandles that are solvable of length $n$. Let $Q$ be $k$-locally reductive and solvable of length $n+1$ and $\alpha=\gamma_n(Q)$ be the $n$-th element of the derived series of $Q$. The quotient $Q^\prime=Q/\alpha$ is $k$-locally reductive and $n$-solvable, so $Q^\prime\in  \dOC_{nk}$. All of the subquandles $[a]_\alpha$ are abelian and $k$-locally reductive.  Therefore, they are in $\dOC_{k}$ and so $Q\in \dOC_{nk+k} \subset\dOC_{\omega}$ by Lemma \ref{extenstionssss}.
\end{proof}
Propositions \ref{reductivity and Engel prop} and ~\ref{solvable doc iff R} 
imply the following statement.
\begin{corollary}
	Let $G$ be a solvable group. Then ${\rm Conj}(G)\in \dOC_{\omega}$ if and only if $G$ is an Engel group.
\end{corollary}

Finally, we characterize the conjugation quandles in the class $\dOC_2$. 
\begin{proposition}\label{2-engel elements -> doc_2}
	Let $G$ be a group and $H\subseteq G$ be a subset of $G$ closed under conjugation. Then ${\rm Conj}(H)\in\dOC_2$ if and only if the elements of $H$ are $2$-Engel elements of $\langle H\rangle$.
\end{proposition}
\begin{proof}
If $x \in H$, then the orbit of $x$ in $H$ is the set $\setof{h^{-1} x h}{h\in \langle H\rangle}$. So, we have
	\begin{displaymath}
	(h^{-1}xh)^{-1}x(h^{-1}xh)=h^{-1} x^{-1} h x h^{-1} x h=[h,_2 x]x.
	\end{displaymath}
Thus, ${\rm Conj}(H) \in   \dOC_2$ if and only if $[h,_2 x]=1$ 
for every $h\in  \langle H \rangle$. 
\end{proof}  
Proposition~\ref{2-engel elements -> doc_2} has the following corollary.
\begin{corollary}
	Let $G$ be a group and $Q={\rm Conj}(G)$. Then $Q\in \dOC_2$ if and only if $G$ is a $2$-Engel group. Moreover, in this case $Q\in \mathcal{R}_3$.
\end{corollary}
\begin{proof} The fact that $Q\in \dOC_2$ if and only if $G$ is a $2$-Engel group follows from Proposition~\ref{2-engel elements -> doc_2}. Every $2$-Engel group is nilpotent of length at most three \cite{levi}. Therefore, $\lmlt(Q)=G/Z(G)$ is nilpotent of length at most $2$ and so, according to Proposition \ref{strong red and inn}, we have $Q\in \mathcal{R}_3$.
\end{proof}

Note that this corollary gives us a bound on the  length of reductivity for a quandle whose orbits of orbits are trivial, at least in the case of conjugation quandles. This observation naturally leads to the following unresolved questions, with which we conclude the paper.

In general, by Corollary \ref{equivalences_for_finite}, if a finite quandle $Q$ is $k$-locally reductive, then it is in $\dOC_j$ for some $j$. Also, if it is in $\dOC_j,$ then it is $i$-reductive for some $i$.
\begin{question} Is it possible to give a bound on $j$ in terms of $k$? On $i$ in terms of $j$?
\end{question}

Looking back at Example \ref{example1}, we see that when $k=2$, then $j$ can be at least $3$; and when $j=3$, $i$ can be at least $5$. Beyond this, little appears to be known.

\end{document}